\documentclass{amsart}
\usepackage{amssymb}
\usepackage{graphicx}
\usepackage{fullpage}
\newtheorem{thm}{Theorem}
\newtheorem{cor}[thm]{Corollary}
\newtheorem{lem}[thm]{Lemma}

\theoremstyle{definition}

\theoremstyle{remark}

\newcommand{\Gal}{\operatorname{Gal}}
\newcommand{\chr}{\operatorname{char}}

\begin{document}

\title[A note on Amitsur's noncrossed product theorem]{A note on Amitsur's noncrossed product theorem}%
\author{Mehran Motiee}%
\address{Faculty of Basic Sciences, Babol Noshirvani University of Technology, Babol, Iran}%
\email{motiee@nit.ac.ir}%


\begin{abstract}
In this note we give a short and elementary proof for a part of Amitsur's noncrossed product theorem.
Our approach does not rely on well-known results of valuation theory. Instead, we employ some
preliminary properties of the unit groups of formal iterated Laurent series division rings.
\end{abstract}
\maketitle
Let $D$ be a finite dimensional division algebra over its center $F$. Recall that
$D$ is called a crossed product if it contains a maximal subfield which is Galois over $F$.
For a finite group $G$, $D$ is said to be a $G$-crossed product if it contains a maximal Galois subfield
$K$ with $\Gal(K/F)\cong G$.
In the other direction, $D$ is said to be a noncrossed product if it does not contain any
maximal subfield that is a Galois extension of $F$. The problem of the existence of a
noncrossed product division algebra was open for many years. The first counterexamples
were given by Amitsur \cite{Am72} in 1972. Other examples of noncrossed products, based on different techniques,
were provided by other authors. For a survey of different kinds of noncrossed products see \cite[Sec. 9.4]{WT15}.
The aim of this note is to give a short and elementary proof for a part of Amitsur's counterexamples. Before
stating our proofs, we give a short outline for Amitsur's method. His examples of noncrossed
products are certain universal
division algebras which are defined as follows: Let $k$ be a ground field and
let $X=\{ x_{ij}^{\left( r\right) }| 1\leq i,j\leq n,r\geq 2\}$ be a set of
independent commuting variables. Let $k[X]$ be the integral domain of polynomials
in all $x_{ij}^{\left( r\right) }$ with coefficients in $k$ and
$k(X)$ the fraction field of $k[X]$. For each $r$, the standard generic
$n$ by $n$ matrices over $k$ are defined by
$\xi ^{\left( r\right) }=\left[ x_{ij}^{\left( r\right) }\right] \in M_{n}\left( F\left[ X\right] \right) \subseteq M_{n}\left( F\left( X\right) \right) $.
The generic matrix algebra over $k$ of degree $n$, by definition, is the
$F$-subalgebra of $M_{n}\left( F\left( X\right) \right)$ that is generated by
$\{  \xi ^{\left( k\right) }| r < n\} $. We denote this algebra by $B_n$.
It is well-known that $B_n$ is a (noncommutative) domain \cite[Prop. 20.5]{Pi82}, so its center is an
integral domain. We define the universal division algebra over $k$ of degree $n$ by
$UD(k,n)=B_n\otimes_{R}L_n$ where $R$ is the center of
$B_n$, and $L_n$ is the fraction field of $R$. Using the theory of PI-rings, one can observe that
$UD(k,n)$ is a division algebra of degree $n$ (cf. \cite[Prop. 20.8]{Pi82}). It is worth
noticing that, for whom who are not familiar with the theory of PI-rings,
in \cite[Chap. 20]{Pi82} a coherent and enough collection of prerequisites for realizing the above proofs is provided. The
main result Amitsur proved in order to show that certain $UD(k,n)$ are noncrossed products is the following:
\begin{thm}\label{t1}
	If $UD(k,n)$ is a $G$-crossed product, then every division algebra of degree $n$
	whose center contains a subfield isomorphic to $k$ is also a $G$-crossed product.
\end{thm}
A proof of Theorem \ref{t1} in its full generality\footnote{In his original paper Amitsur only considers the case $k=\mathbb{Q}$, the field of rational numbers.}
can be found in \cite[p. 417]{Pi82}. Having Theorem \ref{t1} in hand, to prove that $UD(k,n)$ is not a crossed
product it suffices to produce different examples of division algebras over $k$ whose
maximal subfields do not have a common Galois group. The aim of this note is to present a simple
proof with minimal computations for this step in the case $\chr k\neq p$ and $n=p^m$ for some prime $p$ and some $m\geq 3$.
It should be noted here that the most general case (regardless of the algebraic properties of the base field $k$) in which it has been shown that $UD(k,n)$
is not a crossed product is the case $p^3|n$ and $\operatorname{char} k\nmid n$ (see \cite[p. 477]{WT15}).
Unlike previous proofs, our approach does not rely on known results of valuation theory (specifically Hensel's lemma).
We only use the definition of valuations in order to simplify
our presentation. In this sense, our method can be considered elementary.

In what follows, $p$ is always a prime number, $k$ is a field with $\operatorname{char} k\neq p$ and $\omega_n$ denotes a primitive 
$n$-th root of unity. 

In the first step, we recall some well-known facts about the multiplicative groups of Laurent series fields and Laurent 
series division rings.

\textbf{I.}  Let $k(\!(x)\!)$ be the field of Laurent series
over $k$. Let $f(x)=1+a_{1}x+a_2x^2+\dots \in k(\!(x)\!)$.
We show that $f(x)$ has a $p$-th root in $k(\!(x)\!)$.
For this, set $g\left( x\right) =1+b_{1}x+b_2x^2+\dots $
in which $b_n$'s are not known. We wish to find the coefficients of $g(x)$ so that $g(x)^p=f(x)$.
A direct computation shows that
$$
g\left( x\right) ^{p}=1+ pb_{1} x+\left( pb_{2}+\phi _{2}\left( b_{1}\right) \right) x^{2}+\dots
+\left( pb_{n}+\phi _{n}\left( b_{1},\ldots ,b_{n-1}\right) \right) x^{n}+\dots
$$
where each $\phi _{n}\left( b_{1},\ldots ,b_{n-1}\right)$ is a function in $b_1,\dots,b_{n-1}$.
Note that, for our purpose, we have not required the explicit values of $\phi_n$'s.
Comparing coefficients of $f(x)$ and $g(x)^p$ gives

\begin{equation}\label{e}
	a_{1}= p  b_{1},\ a_{2}= p  b_{2}+\phi_{2}\left( b_{1}\right),\ \dots,\ 
	a_{n}= p b_{n}+\phi_{n}\left(b_{1}, \dots, b_{n-1}\right),\ \dots
\end{equation}

Now, since $p$ is invertible in $k$, from \eqref{e},
we may calculate $b_n$'s successively as follows
$$
b_{1}=p^{-1}a_{1},\ b_{2}=p^{-1}\left( a_{2}-\phi _{2}\left( b_{1}\right) \right),\ \dots,\ 
b_{n}=p^{-1}\left( a_{n}-\phi _{n}\left( b_{1},\ldots ,b_{n-1}\right) \right),\ \dots
$$


\textbf{II.} Throughout this item $\omega_p\in k$ and $k$ is $p$-divisible, i.e., $k^*=k^{*p}$ (so $\omega_{p^n}\in k$ for all $n\geq 1$). 
Let $F_{0,0}=k$ and for each $r\geq 1$
let $F_{r,0}=k(\!(x_1,\dots,x_r)\!)$ be the field of iterated Laurent series in variables $x_1,\dots,x_r$ over $k$.
For $m\geq 1$ let $F_{r,m}=k(\!(x_1^{1/p^m},\dots,x_r^{1/p^m})\!)$. Note that
$F_{r,m}=F_{r,0}(x_1^{1/p^m},\dots,x_r^{1/p^m})$ and hence $F_{r,m}/F_{r,0}$ is a finite abelian Galois extension.
Set $F_r=\cup_{m\geq 0}F_{r,m}$. Our purpose is to show that $F_r$ is $p$-divisible. We now proceed by induction on $r$. 
The case $r=0$ is guaranteed by the assumption $k^*=k^{*p}$. Let $F_{r-1}$ be $p$-divisible for 
$r>0$. If $a\in F_r$ then $a$ lies in some $F_{r,m}$ and hence we can write it in the form 
$a=b_0x_r^{k/p^m}\left(1+b_1x_r^{1/p^m}+b_2x_r^{(1/p^m)2}+\dots\right)$, where $b_i\in F_{r-1}$ for all $i\geq 0$.
By induction hypothesis and \textbf{I}, both $b_0$ and $1+b_1x_r^{1/p^m}+b_2x_r^{(1/p^m)2}+\dots$ are $p$-th roots in $F_r$. Furthermore, $x_r^{k/p^m}=x_r^{\left(k/p^{m+1}\right)p}$ showing that $x_r^{k/p^m}$ is also 
a $p$-th root in $F_r$. So $a$ is a $p$-th
root in $F_r$ and the result follows.

\textbf{III.} 
Let $k$ be an arbitrary field and let $n_1,\dots,n_r$ be integers (not necessarily distinct) with $n_i\geq 2$ for all $i$.
let
$$
n=n_1\dots n_r\quad\text{and}\quad m=\operatorname{lcm}(n_1,\dots,n_r).
$$

Assume that $k$ contains a primitive $m$-th root of unity $\omega_m$, and let
$\omega_{n_i}=\omega_m^{m/n_i}$, for all $i=1,\dots,r$; so $\omega_{n_i}$ is a primitive
$n_i$-th root of unity.
Let $\Delta_{2r}(k; n_1,\dots,n_r)=k(\!(x_1,y_1,\dots,x_{r},y_{r})\!)$
be the division ring of the formal iterated Laurent series in variables
$x_1,y_1,\dots,x_{r},y_{r}$ with the multiplication relations:
$$
x_ia=ax_i\quad,\quad y_ia=ay_i\quad \text{for}\ a\in k
$$
$$
x_ix_j=x_jx_i\quad,\quad y_iy_j=y_jy_i
$$
$$
x_iy_j=y_jx_i\quad \text{for}\ i\neq j
$$
$$
x_iy_i=\omega_{n_i} y_ix_i.
$$
Recall that $\Delta_{2r}(k; n_1,\dots,n_r)$ is a division algebra of degree $n$ and center
$F=k(\!(x_1^{n_1},y_1^{n_1},\dots,x_{r}^{n_r},y_{r}^{n_r})\!)$  (cf. \cite[p. 100]{J75}).
From now on, for abbreviation, we write $\Delta$ instead of $\Delta_{2r}(k; n_1,\dots,n_r)$.
Recall that $\Delta$ has a natural valuation with the value group $\Gamma=\oplus_{i=1}^{2r}\mathbb{Z}$ given by
$$
v\left(\sum_{i_1,\dots,i_{2r}}a_{i_1,\dots,i_{2r}}x_1^{i_1}y_1^{i_i}\dots x_r^{i_{2r-1}}y_r^{i_{2r}}\right)=
\min\{(i_1,\dots,i_{2r})|a_{i_1,\dots,i_{2r}}\neq 0\}.
$$
To simplify our notations, for $f\in \Delta^*$ we write $f=a_{\alpha}x^{\alpha}(1+d)$ where
$\alpha=v(f)$, $a_{\alpha}\in k^*$ and $d\in M_{\Delta}$, where $M_{\Delta}=\{g\in\Delta^*|v(g)>0\}$.
So we have $\Delta^*=H(1+M_{\Delta})$ where
$H=\{ax^{\alpha}|a\in k^*, \alpha\in \Gamma\}$.
From the multiplication rule of
$\Delta$ it follows that $x^{\alpha}x^{\beta}=\omega_m^{e}x^{\alpha+\beta}$
for some integer $e$. This gives $H'=\langle\omega_m\rangle$, where $H'$ is the commutator group of $H$. Since
$1+M_{\Delta}$ is a normal subgroup of $\Delta^*$, we conclude that
\begin{equation}\label{e2}
	\Delta'\subseteq H'(1+M_{\Delta})=\langle\omega_m\rangle(1+M_{\Delta}).
\end{equation}

The following lemma is required to provide our proof. 
\begin{lem}\label{l1}
	Let $F$ be a field whose multiplicative group is
	$p$-divisible and let $\omega_p\in F$. Let $L$ be a subfield of $F$ and $K$ be
	a finite extension of $L$ contained in some algebraic closure of $F$. If $K/L$ is Galois and $[K:L]=p^m$ for some $m\geq 0$
	then $K\subseteq F$.
\end{lem}
\begin{proof}
	We first observe that by Galois theory $FK/F$
	is Galois and $\Gal(FK/F)\cong \Gal(K/K\cap F)$.
	Since $\Gal (K/K\cap F)$ is a subgroup of $\Gal(K/L)$, we have
	$|\Gal (FK/F)|=p^r$ for some nonnegative integer $r$.
	If $r\neq 0$ then by the theory of $p$-groups
	$\Gal(FK/F)$ contains a normal subgroup $H$ such that $|H|=p^{r-1}$.
	Let $M$ be the fixed field of $H$. It follows from the
	fundamental theorem of Galois theory that $M/F$ is Galois with $[M:F]=p$.
	So $M/F$ is cyclic which is impossible, because by Kummer Theory conditions
	$F^*=F^{*p}$ and $\omega_p\in F $ imply that $F$ does not have
	a cyclic extension of degree $p$. So $r=0$. Therefore, $FK=F$ and hence $K\subseteq F$, as desired.
\end{proof}

\begin{cor}\label{lem3}
	Let $K$ be finite Galois extension of $F_{r,0}$ such that $[K:F_{r,0}]$ is a power of $p$. Then
	$\Gal(K/F_{r,0})$ is abelian.
\end{cor}
\begin{proof}
	Since $F_r$ is $p$-divisible, from Lemma \ref{l1} it follows that $K$ is a subfield of $F_r$. So $K$ lies in some $F_{r,m}$.
	Since $F_{r,m}$ is an abelian extension of $F_{r,0}$ we conclude that $K/F_{r,0}$ is abelian.
\end{proof}

\begin{thm}
	Let $k$ be an arbitrary field with $\chr k\neq p$.
	If $n\geq 3$ then $UD(k,p^n)$ is noncrossed product.
\end{thm}
\begin{proof}
	On the contrary, suppose that $UD(k,p^n)$ is a $G$-crossed product.
	Consider the division algebra $D_1=\Delta_{2n}(k_{\operatorname{sep}};\underbrace{p,\dots, p}_n)$ where
	$k_{\operatorname{sep}}$ is a separable closure of $k$.
	By Theorem \ref{t1}, $D_1$ is a $G$-crossed product and hence it contains a maximal subfield
	$K$ with $\Gal(K/F)\cong G$. From Corollary \ref{lem3}, it follows that
	$G$ is abelian. We claim that $G\cong\oplus_{j=1}^n \mathbb{Z}_{p}$.
	To prove our claim, it suffices to show that each $\sigma\in G$ has order $p$. Let, there exists an element
	$\sigma\in G$ such that $O(\sigma)>p$. Since $G$ is abelian, we can choose a $\tau\in G$ with $O(\tau)=p^2$. Let
	$L$ be the fixed field of $\tau$. By Centralizer Theorem \cite[p. 42]{Dra83}, $E=C_{D_1}(L)$ is an $L$-central division algebra
	of degree $p^2$ and $K$ is a maximal subfield of $E$. Furthermore, we have $\Gal(K/L)=\langle \tau\rangle$.
	Now, we observe that $N_{K/L}(\omega_{p^2})=\left(\omega_{p^2}\right)^{p^2}=1$, where $N_{K/L}$ is the norm map from
	$K^*$ to $L^*$ (note that $\omega_{p^2}\in k_{\operatorname{sep}}\subseteq L$).
	So by Hilbert's satz 90 we can write $\omega_{p^2}=\tau(a)a^{-1}$ for some $a\in K^*$. On the other hand, by
	Skolem-Noether Theorem (\cite[p. 39]{Dra83}), there is a $t\in E^*$ such that $\tau(a)=tat^{-1}$. Hence
	$\omega_{p^2}=tat^{-1}a^{-1}\in E'\subseteq D_1'$, contrary to \eqref{e2} (recall that by \eqref{e2} we have 
	$D_1'\subseteq\langle\omega_p\rangle(1+M_{D_1})$).
	This is our claim. So $\operatorname{rank}(G)=n\geq 3$. Now, let
	$D_2=\Delta_{2}(k_{\operatorname{sep}}; p^n)$.
	Let $P=k_{\operatorname{sep}}(\!(x_1^{p^n},y_1^{p^n})\!)$ which is the center of $D_2$.
	Theorem \ref{t1} shows that $D_2$ is again a $G$-crossed product. Let $N$ be a maximal subfield of $D_2$ such
	that $\Gal(N/P)\cong G$. Since $G$ is abelian and $\omega_{p^n}\in k_{\operatorname{sep}}$, from Kummer Theory
	it follows that $G$ is isomorphic to a subgroup of $P^*/P^{*p^n}$. But,
	each $a\in P^*$ may be written in the form $a=ux_1^{rp^n}y_1^{sp^n}\left(1+v_1y_1^{p^n}+v_2y_1^{2p^n}+\dots\right)$ where 
	$u\in k_{\operatorname{sep}}^*$ and $v_i\in k_{\operatorname{sep}}(\!(x_1^{p^n})\!)$ for all $i\geq 1$.
	Because $u$ and $1+v_1y_1^{p^n}+v_2y_1^{2p^n}+\dots$ are a $p^n$-th root (see \textbf{I}), it follows that $P^*/P^{*p^n}$
	is generated by the images of $x_1^{p^n}$ and $y_1^{p^n}$.
	In particular, we have $\operatorname{rank}(G)\leq 2$. This contradiction shows that
	$UD(k,p^n)$ is not a crossed product.
\end{proof}


\bibliographystyle{plain}

\end{document}